\documentclass[11pt]{article}
\usepackage{amsmath,amssymb,amsthm,enumerate,graphicx,bbm,appendix,booktabs,enumitem,verbatim,subfigure}
\usepackage{multirow}
\usepackage{rotating}   
\usepackage{makecell}   
\newcolumntype{Y}{>{\raggedright\arraybackslash}X} 

\usepackage{array}
\usepackage{hyperref}
\usepackage{tikz}
\usetikzlibrary{positioning,arrows.meta,fit,calc}
\tikzset{
  block/.style={draw,rounded corners,thick,align=center,
                minimum width=3.2cm,minimum height=1.15cm},
  smallblock/.style={draw,rounded corners,thick,align=center,
                minimum width=3.0cm,minimum height=1.05cm},
  line/.style={-Latex,very thick},
  sig/.style={-Latex,thick},
  coupling/.style={dashed,-{Stealth[length=2.4mm]},thick},
  note/.style={font=\footnotesize}
}
\usepackage{float}
\def\fskip#1{}

\newtheorem{theorem}{Theorem }

\newtheorem{lemma}[theorem]{Lemma}

\def\be{\begin{enumerate}}
\def\ee{\end{enumerate}}

		\def\bkE{{\rm I\kern-.17em E}}
		\def\bk1{{\rm 1\kern-.17em l}}
		\def\bkD{{\rm I\kern-.17em D}}
		\def\bkR{{\rm I\kern-.17em R}}
		\def\bkP{{\rm I\kern-.17em P}}
		\def\bkY{{\bf \kern-.17em Y}}
		\def\bkZ{{\bf \kern-.17em Z}}

\title{Decentralized Disturbance Rejection Control of Triangularly
Coupled Loop Thermosyphon System}
\author{Novel~ Kumar ~Dey and Yan ~Wu\thanks{Novel and Yan are with Department of Applied Mathematics and Department of Mathematical Sciences, respectively at University of Arizona, Tucson, USA and Georgia Southern University, USA. They are reachable at ({\tt
		ndey@arizona.edu,yan@georgiasouthern.edu}).}}
\begin{document}
\maketitle
\thispagestyle{empty}
\pagestyle{empty}
\begin{abstract}
In this paper, we investigate the stability of a triangularly coupled triple loop thermosyphon system with momentum and heat exchange at the coupling point as well as the existence of disturbances. The controller consists of a single, local state feedback. From the stability analysis, we obtain explicit bounds on the feedback gains, which depend on the Rayleigh numbers and the momentum coupling parameter, but independent of the thermal coupling parameter. The existence of the stability bounds allows us to design decentralized adaptive controllers to automatically search for the feasible gains when the system parameters are unknown. In the case of existing disturbances in the system, we approximate the disturbances via an extended state observer for the purpose of disturbance rejection. Numerical results are given to demonstrate the performance of the proposed decentralized disturbance rejection controller design.
	\end{abstract}
\textbf{keywords:} {Loop thermosyphon; triangularly coupled; linear coupling; decentralized control; adaptive control; extended state observer; Luenberger observer; active disturbance rejection control; Lyapunov stability}
\section{Introduction}\label{sec:intro}
Large networks of interconnected dynamical systems are ubiquitous in modern engineering applications \cite{ref1}. For instance, the development of cross-directional control systems for paper machines, power distribution systems, automated highways, formations
of unmanned aerial vehicles, and arrays of microcantilevers for massively parallel data storage, to name a few. Large-scale dynamical systems usually consist of a few
compartmental subsystems that interact with each other through different means of communication. In the context of nonlinear dynamical systems, a large-scale system is characterized by multiple positive Lyapunov exponents or known as hyperchaotic systems. These interconnected dynamical systems are of theoretical interest because they pose new challenges for analysis and control design. In this paper, we consider a triangularly coupled thermosyphon loop system, see Figure \ref{fig1}, in which the loops are heated from below and cooled from above. There exist momentum and heat couplings between the adjacent loops. In other words, the flows in each loop interfere with each other due to the heat and mass exchanges at the contact area. Such interference intensifies dramatically when the external heat increases. The flows evolve from heat conductions, steady convective flows to chaos as the external heat applied to the loops increases.
\begin{figure}[H]
\centering
\includegraphics[width=0.45\textwidth]{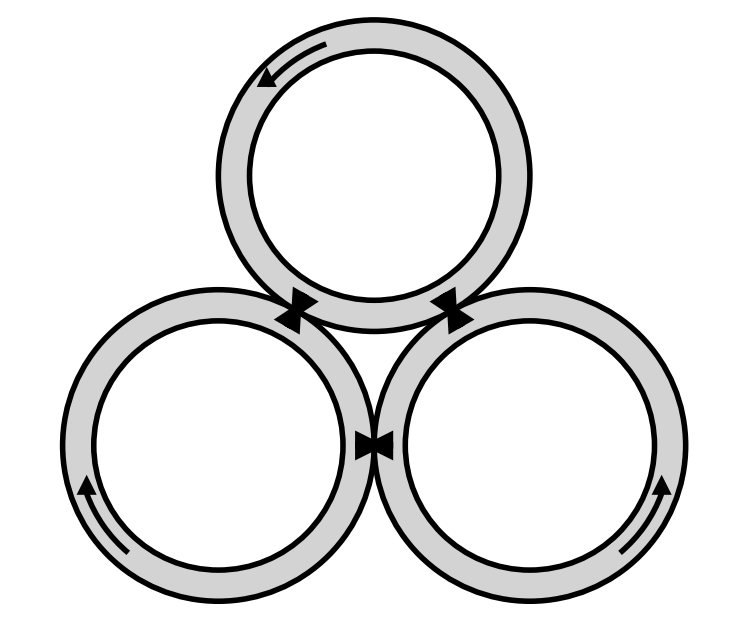}
\caption{Triangularly coupled loop thermosyphon system.}
\label{fig1}
\end{figure}
The mathematical model for a laterally coupled double loop system was developed in \cite{ref2}. We applied the same modeling process to the triangularly coupled triple loop system and obtained the following governing equations
\begin{align}\label{Eqn1.1}
\begin{cases}
& \dot{x}_1=p\left\{\left(y_1-x_1\right)-\gamma_1\left(x_1-x_2\right)-\gamma_2\left(x_1-x_3\right)\right\} \\
& \dot{y}_1=R_1 x_1 -x_1 z_1 -u_1\\
& \dot{z}_1=x_1 y_1-z_1-\eta_1\left(z_1-z_2\right)-\eta_2\left(z_1-z_3\right) \\
& \dot{x}_2=p\left\{\left(y_2-x_2\right)-\gamma_1\left(x_2-x_1\right)-\gamma_3\left(x_2-x_3\right)\right\} \\
& \dot{y}_2=R_2 x_2 -x_2 z_2 -u_2\\
& \dot{z}_2=x_2 y_2-z_2-\eta_1\left(z_2-z_1\right)-\eta_3\left(z_2-z_3\right) \\
& \dot{x}_3=p\left\{\left(y_3-x_3\right)-\gamma_2\left(x_3-x_1\right)-\gamma_3\left(x_3-x_2\right)\right\} \\
& \dot{y}_3=R_3 x_3 -x_3 z_3 -u_3\\
& \dot{z}_3=x_3 y_3-z_3-\eta_2\left(z_3-z_1\right)-\eta_3\left(z_3-z_2\right),
\end{cases}
\end{align}
where $x$-states represent the fluid velocity, the $y,z$-states correspond to the horizontal and vertical temperature differences, respectively. The key parameters are the Rayleigh numbers $R_{1,2,3}$, which are proportional to the external heat applied to the loop; $\gamma_{_{1,2,3}}$ are proportional to the momentum coupling intensity between the adjacent loops, and $\eta_{_{1,2,3}}$ correspond to the thermal coupling. Both coupling intensity parameters are between 0
and 1 \cite{ref2}. The $u_{_{1,2,3}}$ are the decentralized controllers, which is given by 
\begin{align*}
    u_i= -k_i(y_i-y_{r_i}), \qquad i=1,2,3,
\end{align*}
where $k_{1,2,3}$ are the positive state feedback gains and $y_{r_{1,2,3}}$ represent the reference inputs or tracking signals. With large Rayleigh numbers, the triple-loop system \eqref{Eqn1.1} is hyperchaotic because it has two positive Lyapunov exponents based on our computation. The system model \eqref{Eqn1.1} consists of three sets of differential equations, corresponding to the three thermosyphon loops, respectively. In each set, the $x$-equation governs the average velocity of the fluid flow, the $y$-equation prescribes the dynamics of the horizontal temperature difference across the loop, and the vertical temperature difference state satisfies the $z$-equation. There exists momentum and thermal exchanges between adjacent loops through the contact area. This is reflected in \eqref{Eqn1.1} through the coupling terms involving states from the neighboring loops. Our control objective is to stabilize all nine states to the equilibrium point while high external heat is applied to the loops (large Rayleigh numbers). We use the fourth order Runge-Kutta method to carry out all the numerical simulations. The difference between the model equations \eqref{Eqn1.1} and the dynamical equations for the laterally coupled four-loop system in \cite{ref15} is that each of the $x$-equation and $z$-equation has two coupling terms due to the physical interconnection between the loops, whereas these two-coupling terms only exist in the dynamical equations for the middle loops in \cite{ref15}. This subtle difference does not bring about major discrepancies in the dynamical behavior after we compare their simulation results, which are actually quite similar. However, there exists a unique symmetry in the stability bounds on the feedback gains in the three-loop system that is not observed from the bounds obtained in \cite{ref14} and \cite{ref15}. This will be discussed in detail in section \ref{sec2}. To connect simulations with measurable quantities in an experiment, we summarize below the notation and physical meaning of all the core parameters in \eqref{Eqn1.1}. Details of the non-dimensionalization and calculation formulas for the parameters can be found in \cite{ref2}. Below is the nomenclature and physical parameter definitions:
\begin{tabular}{p{0.18\linewidth}p{0.75\linewidth}}
$R_i$ & Rayleigh number of loop $i$, proportional to external heat applied to the loop. \\
$\gamma_i$ & Momentum-coupling coefficient between adjacent loops at a shared junction; proportional to the intensity of momentum exchange. \\
$\eta_i$ & Thermal coupling parameter between adjacent loops; proportional to the intensity of thermal exchange . \\
$x_i$ & Average fluid velocity state in loop $i$. \\
$y_i$ & Horizontal temperature-difference state in loop $i$. \\
$z_i$ & Vertical temperature-difference state in loop $i$. \\
$u_i$ & Local control input applied to loop $i$. \\
$k_i$ & Proportional feedback gain in $u_i=-k_i\,y_i$.\\
\end{tabular}\\
 
For an interconnected dynamical system such as \eqref{Eqn1.1} as well as many applied control problems of large-scale systems, decentralized control has been the forefront runner in controller design. In this paradigm, a local controller is constructed for each subsystem using only the local state for feedback purposes. The merits of decentralized control go
beyond simplification of the controller structure, it also increases reliability, scalability,
and flexibility, and lowers the cost and complexity \cite{ref3,ref4}. Decentralized control is also advantageous in that the system can tolerate faults and failures in some components without affecting the whole system \cite{ref5}. In many practical situations, there are
uncertainties associated with the system due to model simplification, unmodeled dynamics, measurement errors, truncation errors, and unpredictable external inputs to the
system. When it comes to large scale interconnected dynamical systems with
uncertainties, decentralized robust control plays an important role in front of these challenges, such as the design of decentralized state feedback control in the presence of
uncertainties and feedback delays in the interconnected systems \cite{ref6, ref7}. Many well-known control techniques are implemented with decentralized control over large scale systems.
This includes, decentralized adaptive backstepping control for a large class of interconnected systems with unknown interactions between their subsystems \cite{ref8}, output feedback control with nonlinear coupling \cite{ref9}, optimal control of interconnected systems via block pulse functions parameterization \cite{ref10}, and more. Another integral component
for robust decentralized control systems is the reliable state observers \cite{ref11} when the states are unknown. Some robust state observers are proposed for large scale systems with uncertainties \cite{ref12}, and completely decentralized state observers are introduced in
\cite{ref13}. The work related to this paper can be found in \cite{ref14}, where the adaptive control of laterally coupled two-loop system using proportional state feedback control is studied.
The uncertainties therein are estimated via wavelet neural networks. The results are extended to high-dimensional four-loop systems in \cite{ref15}, in which much more sophisticated stability bounds on the control gains are derived. 

As far as the system \eqref{Eqn1.1} is concerned, the control objective is to stabilize the chaotic system by using a minimum number of decentralized controllers. Our controllability
analysis reveals that only three controllers are needed, which depends on the local $y$-state of the system. In other words, the proportional controllers $u_i$ in \eqref{Eqn1.1} are of the form: $u_i=-k_iy_i, i=1,2,3$. Another major advantage of using the $y$-state for feedback purposes is that it is corresponding to the horizontal temperature difference across the loop, which is easily measurable. In terms of measurement and data acquisition through sensors, the $y$-state samples are the only data collected for the controller design, including the extended state observer. The main contributions of this paper are in two-fold:
\begin{itemize}
\item [(i)] prove the existence of a lower bound on the feedback gain $k_i$ that guarantees the global stability of
the model-based decentralized control system \eqref{Eqn1.1}, and
\item[(ii)] building upon the proportional control, we design active disturbance rejection controllers (ADRC) to eliminate the impact of disturbances so that the triangularly coupled triple loop system with
uncertainties remains globally asymptotically stable.
\end{itemize}
The rest of the paper is organized as follows: we present explicit bounds that depend on the key system parameters on the feedback gains in section \ref{sec2}. In the same section, we present results on the adaptive gain search when the system parameters are unknown. In section \ref{sec3}, we consider a more practical system with uncertainties originating from dynamical modeling and external disturbances. Such disturbances are reconstructed through decentralized extended state observers (ESO) so that they can be eliminated through the ADRC approach. We also
prove the asymptotic stability of the ADRC system in the section along with numerical simulations, followed by some conclusive remarks in section \ref{sec4}


\section {Stability Bounds on Feedback Gains}\label{sec2} 
Our goal is to develop a robust adaptive control system to stabilize the flows in an interconnected triangulated coupled system. More specifically, we consider the following $y$-state feedback system after setting $y_{r_i}=0$, $i=1,2,3,$ in \eqref{Eqn1.1}
\begin{equation}\label{Eqn2.1}
\begin{cases}
& \dot{x}_1=p\left\{\left(y_1-x_1\right)-\gamma_1\left(x_1-x_2\right)-\gamma_2\left(x_1-x_3\right)\right\} \\
& \dot{y}_1=R_1 x_1 -x_1 z_1 -k_1y_1\\
& \dot{z}_1=x_1 y_1-z_1-\eta_1\left(z_1-z_2\right)-\eta_2\left(z_1-z_3\right) \\
& \dot{x}_2=p\left\{\left(y_2-x_2\right)-\gamma_1\left(x_2-x_1\right)-\gamma_3\left(x_2-x_3\right)\right\} \\
& \dot{y}_2=R_2 x_2 -x_2 z_2 -k_2y_2\\
& \dot{z}_2=x_2 y_2-z_2-\eta_1\left(z_2-z_1\right)-\eta_3\left(z_2-z_3\right) \\
& \dot{x}_3=p\left\{\left(y_3-x_3\right)-\gamma_2\left(x_3-x_1\right)-\gamma_3\left(x_3-x_2\right)\right\} \\
& \dot{y}_3=R_3 x_3 -x_3 z_3 -k_3y_3\\
& \dot{z}_3=x_3 y_3-z_3-\eta_2\left(z_3-z_1\right)-\eta_3\left(z_3-z_2\right) .
\end{cases}
\end{equation}
In this work, we examine the controllability properties under various state-feedback configurations. Among these options, feedback through the $y$-state emerges as particularly advantageous from the stabilizability aspect, primarily due to its ability to ensure positivity of feedback gains. Furthermore, practical implementation of the y-state feedback is straightforward, as these states correlate directly with fluid temperature measurements. The main goal here is to establish explicit conditions on the feedback parameters $k_{1,2,3}$ to guarantee stabilization of system \eqref{Eqn2.1} at its equilibrium.

To begin our analysis, we consider the following symmetric matrix $A$, which explicitly depends on the feedback gains $k_{1,2,3}$, as given below:
\begin{equation}\label{mat1}
A=\left[\begin{smallmatrix}
R_1(1+\gamma_1+\gamma_2) & -R_1 & 0 & -(R_1+R_2)\gamma_1/2 & 0 & 0 & -(R_1+R_3)\gamma_2/2 & 0 & 0\\
-R_1 & k_1 & 0 & 0 & 0 & 0 & 0 & 0 & 0\\
0 & 0 & (1+\eta_1+\eta_2) & 0 & 0 & -\eta_1 & 0 & 0 & -\eta_2\\
-(R_1+R_2)\gamma_1/2 & 0 & 0 & R_2(1+\gamma_1+\gamma_3) & -R_2 & 0 & -(R_2+R_3)\gamma_3/2 & 0 & 0\\
0 & 0 & 0 & -R_2 & k_2 & 0 & 0 & 0 & 0\\
0 & 0 & -\eta_1 & 0 & 0 &  (1+\eta_1+\eta_3) & 0 & 0 & -\eta_3\\
-(R_1+R_3)\gamma_2/2 &  0 & 0 & -(R_2+R_3)\gamma_3/2 & 0 & 0 & R_3(1+\gamma_2 + \gamma_3) & -R_3 & 0\\
0 & 0 & 0 & 0 & 0 & 0 & -R_3 & k_3 & 0\\
0 & 0 & -\eta_2 & 0 & 0 & -\eta_3 & 0 & 0 & (1+\eta_2+\eta_3)
\end{smallmatrix}\right]
\end{equation}
This matrix plays a critical role in analyzing the stability of system \eqref{Eqn2.1}. It is noted that, within matrix \eqref{mat1}, the feedback gains $k_i$ represent the sole adjustable parameters, whereas the remaining system parameters remain fixed. Our objective is to ensure that the matrix $A$ is positive definite, which allows us to derive explicit conditions on the gains that render every principal minor of $A$ strictly positive. To accomplish this, we adopt specific balanced formulations involving pairs of Rayleigh numbers to facilitate the derivation of precise stability criteria.
\begin{equation}\label{Eqn2.2}
   \begin{cases}
      &\psi_1=R_1 R_2 \{(1 + \gamma_1+\gamma_2)(1 +\gamma_1+\gamma_3)-\gamma_1^2\} - (R_1 - R_2)^2 \gamma_1^2/4 >0\\
      &\psi_2=R_1 R_3 \{(1 + \gamma_1+\gamma_2)(1 +\gamma_2+\gamma_3)-\gamma_2^2\} - (R_1 - R_3)^2 \gamma_2^2/4>0\\
      &\psi_3=R_2 R_3 \{(1 + \gamma_1+\gamma_3)(1 +\gamma_2+\gamma_3)-\gamma_3^2\} - (R_2 - R_3)^2 \gamma_3^2/4>0.
    \end{cases}
\end{equation}
We assert the positivity of these expressions by leveraging the close numerical proximity of the Rayleigh numbers, combined with the condition $\gamma_i \ll 1$, a typical scenario arising in loop thermosyphon applications. We define $A_i$ as the \textit{$i$-th} principal minor of the matrix $A$. The following lemmas establish the positive definiteness of $A$.

\begin{lemma}\label{combined_lemma1&2}
The first three principal minors of matrix $A$ are positive if 
\begin{align}\label{Eqn2.3}
    k_1 > \frac{R_1}{1 + \gamma_1+\gamma_2}.
\end{align}
\begin{proof} The positivity of $A_1$ is immediate, as it directly follows from the definition
\begin{align*}
    A_1 = R_1(1 + \gamma_1+\gamma_2) > 0,
\end{align*}
given that $R_1 > 0$ and $\gamma_i's > 0$. For $A_2$, we compute the determinant explicitly
\begin{align*}
A_2 = \begin{vmatrix}
R_1(1 + \gamma_1+\gamma_2) & -R_1 \\
-R_1 & k_1
\end{vmatrix} = R_1 k_1 (1 + \gamma_1+\gamma_2) - R_1^2.
\end{align*}
Note that, requiring $A_2 > 0$ yields
\begin{align*}
    k_1 > \frac{R_1}{1 + \gamma_1+\gamma_2}.
\end{align*}
Similarly for $A_3$, we have
\begin{align*}
    A_3 = \begin{vmatrix}
R_1(1 + \gamma_1+\gamma_2) & -R_1 & 0 \\
-R_1 & k_1 & 0 \\
0 & 0 & (1 + \eta_1+\eta_2)
\end{vmatrix} = (1 + \eta_1+\eta_2) \left[ R_1 k_1(1+\gamma_1+\gamma_2) - R_1^2 \right] > 0,
\end{align*}
which imposes the same condition on $k_1$ we obtained from $A_2$ as the bound indicates that the thermal coupling parameters do not destabilize the control system.
\end{proof}
\end{lemma}

\begin{lemma}\label{lemma2.2}
The principal minor $A_4$ is positive if
\begin{equation} \label{Eqn2.4}
k_1 > \frac{R_1^2 R_2(1 + \gamma_1+\gamma_3)}{R_1 R_2(1 + \gamma_1+\gamma_2)(1 + \gamma_1+\gamma_3) - (R_1 + R_2)^2 \gamma_1^2/4}.
\end{equation}
\begin{proof} Consider the 4×4 principal submatrix of $A$, whose determinant is
\begin{align*}
    A_4 = \begin{vmatrix}
R_1(1 + \gamma_1+\gamma_2) & -R_1 & 0 & -(R_1 + R_2)\gamma_1/2 \\
-R_1 & k_1 & 0 & 0 \\
0 & 0 & (1 + \eta_1+\eta_2) & 0 \\
-(R_1 + R_2)\gamma_1/2 & 0 & 0 & R_2(1 + \gamma_1+\gamma_3)
\end{vmatrix}.
\end{align*}
This determinant simplifies to 
\begin{align*}
    A_4 = (1 + \eta_1+\eta_2)\left[\{R_1R_2 (1 + \gamma_1+\gamma_2)(1 + \gamma_1+\gamma_3)-(R_1+R_2)^2\gamma_1^2/4\}k_1 - R_1^2 R_2(1 + \gamma_1+\gamma_3) \right].
\end{align*}
By making $A_4 > 0$, one has \eqref{Eqn2.4} provided the denominator of \eqref{Eqn2.4} is positive. This is true because the denominator can be rewritten into $\psi_1$ in \eqref{Eqn2.2}.
\end{proof}
\end{lemma}

\begin{lemma}\label{lemma2.3}
The principal minors $A_5$ and $A_6$ are both positive if the feedback gain $k_2$ satisfies
\begin{equation} \label{Eqn2.5}
    k_2 > \frac{R_1 R_2^2 \{(1 + \gamma_1+\gamma_2) k_1 - R_1\}}{k_1 \psi_1 - R_1^2 R_2 (1 + \gamma_1+\gamma_3)},
\end{equation}
where $\psi_1$ is defined in \eqref{Eqn2.2}.   
\end{lemma}
\begin{proof}
    We start by evaluating the 5th-order principal minor
    \begin{align*}
        A_5 = \begin{vmatrix}  R_1(1 + \gamma_1+\gamma_2) & -R_1 & 0 & -(R_1 + R_2)\gamma_1/2 & 0 \\  -R_1 & k_1 & 0 & 0 & 0 \\  0 & 0 & (1 + \eta_1+\eta_2) & 0 & 0 \\  -(R_1 + R_2)\gamma_1/2 & 0 & 0 & R_2(1 + \gamma_1+\gamma_3) & -R_2 \\  0 & 0 & 0 & -R_2 & k_2  \end{vmatrix}.
    \end{align*}
The following inequality is obtained from $A_5 > 0$,
\begin{align*}
    k_2 > \frac{R_1 R_2^2 \{(1 + \gamma_1+\gamma_2) k_1 - R_1\}}{k_1 \psi_1 - R_1^2 R_2 (1 + \gamma_1+\gamma_3)},
\end{align*}
establishing the lower bound on $k_2$. The denominator is positive due to \eqref{Eqn2.3} and the positivity of $\psi_1$ in \eqref{Eqn2.2}. Now consider the 6th-order principal minor, we have
\begin{align*}
        A_6 &= \begin{vmatrix}  R_1(1 + \gamma_1+\gamma_2) & -R_1 & 0 & -(R_1 + R_2)\gamma_1/2 & 0 &0\\  -R_1 & k_1 & 0 & 0 & 0 &0\\  0 & 0 & (1 + \eta_1+\eta_2) & 0 & 0 & -\eta_1 \\  -(R_1 + R_2)\gamma_1/2 & 0 & 0 & R_2(1 + \gamma_1+\gamma_3) & -R_2 &0 \\  0 & 0 & 0 & -R_2 & k_2 & 0\\
        0 & 0 & -\eta_1 & 0 & 0 & (1+\eta_1+\eta_3)\end{vmatrix}.
    \end{align*}
This evaluates to
\begin{align*}
   A_6= \frac{(1+2\eta_1+\eta_2+\eta_3+\eta_1\eta_2+\eta_1\eta_3+\eta_2\eta_3)}{(1+\eta_1+\eta_2)}A_5>0,
\end{align*}
which produces the same lower bound on $k_2$ as that from $A_5>0$.
\end{proof} 

It is important to note that, while the principal minors $A_1$ through $A_6$ involve only a subset of the system parameters and feedback gains, the next three principal minors,$A_7$, $A_8$, and $A_9$, are of particular interest. These minors involve all system parameters, namely $R_i's, \gamma_i's,$ and $\eta_i's$, and hence play an important role in ensuring that matrix $A$ is positive definite.  We first introduced the following shorthand notations:
\begin{align*}
\gamma_{_{ij}} = 1 + \gamma_i + \gamma_j, \quad  
 R_{_{ij}} = R_i + R_j, \quad i,j \in \{1,2,3\}.
\end{align*}

\begin{lemma}\label{lemma2.4}
    The principal minor $A_7$ is positive if the feedback gains $k_1$ and $k_2$ satisfy
\begin{align}\label{Eqn2.6}
 k_1 > \frac{R_1^2 \left[R_2 R_3 \gamma_{_{13}} \gamma_{_{23}} - R^2_{23} \gamma_3^2/4 \right]}{%
\begin{array}{l}
R_3 \psi_1 \gamma_{_{23}} - R_{12}R_{23}R_{31} \gamma_{_1} \gamma_{_2} \gamma_{_3}/4 - R_2 \gamma_{_{13}} R_{31}^2 \gamma_2^2/4 - R_1 \gamma_{_{12}} R_{23}^2 \gamma_3^2/4 \end{array}
}
\end{align}
and
\begin{align} \label{Eqn2.7}
k_2 > \frac{R_2^2 \psi_2k_1-R_1^2R_2^2R_3 \gamma_{_{23}}}{%
\begin{array}{l}
R_3 \psi_1 k_1 \gamma_{_{23}} - R_{12}R_{23}R_{31} k_1 \gamma_{_1} \gamma_{_2} \gamma_{_3}/4-R_1^2\psi_3 - R_2 k_1 \gamma_{_{13}} R_{31}^2 \gamma_2^2/4- R_1k_1 \gamma_{_{12}}R_{23}^2 \gamma_3^2/4
\end{array}
}.
\end{align}
\begin{proof} We consider the $7^{th}$ order principal minor of matrix $A$
   \begin{align*}
       A_7 &= \begin{vmatrix}  R_1 \gamma_{_{12}} & -R_1 & 0 & -R_{12} \gamma_{_{1}}/2 & 0 &0 & -R_{13} \gamma_{_{2}}/2\\
       -R_1 & k_1 & 0 & 0 & 0 &0 & 0\\  0 & 0 & (1 + \eta_1+\eta_2) & 0 & 0 & -\eta_1 & 0 \\ 
       - R_{12} \gamma_{_{1}}/2 & 0 & 0 & R_2 \gamma_{_{13}} & -R_2 &0  & - R_{23} \gamma_{_{3}}/2\\
       0 & 0 & 0 & -R_2 & k_2 & 0 & 0\\
        0 & 0 & -\eta_1 & 0 & 0 & (1+\eta_1+\eta_3) & 0\\
        - R_{13} \gamma_{_{2}}/2& 0 & 0 & - R_{23} \gamma_{_{3}}/2 & 0 & 0 & R_3 \gamma_{_{23}}\end{vmatrix}.
    \end{align*}
Using the determinant expansion, and previous results from $A_6$, we obtain
\begin{align*}
A_7 &= h_1(\eta)\Big[ 
k_2\{R_3 \psi_1 k_1 \gamma_{_{23}} - R_{12}R_{23} R_{31} k_1 \gamma_1 \gamma_2 \gamma_3/4 - R_2 k_1  \gamma_{_{13}} R_{31}^2 \gamma_2^2/4- R_1 k_1 \gamma_{_{12}} R_{23}^2 \gamma_3^2/4- R_1^2 \psi_3\} \\
&- R_2^2 \psi_2 k_1 + R_1^2 R_2^2 R_3 \gamma_{_{23}}
\Big],
\end{align*}
where $h_1(\eta)=(1 + 2\eta_1 + \eta_2 + \eta_3 + \eta_1\eta_2 + \eta_1\eta_3 + \eta_2\eta_3)>0$. Then setting $A_7 > 0$ and solving for $k_2$ yields inequality \eqref{Eqn2.7}. For the inequality to be valid, the denominator of which must be positive, which leads to a new constraint on $k_1$ as given in \eqref{Eqn2.6}.
\end{proof}
\end{lemma}
Having established the lower bounds on the first two feedback gains $k_1$ and $k_2$ that ensure positiveness of the principal minors up to $A_7$, we now turn our attention to the third and final feedback gain, $k_3$. The following lemma presents a lower bound on $k_3$, which  guarantees that the remaining principal minors are positive.

\begin{lemma}\label{lemma2.5}
The principal minors \(A_8\) and \(A_9\) are positive provided  \(k_3\) satisfies
\begin{equation}\label{Eqn2.8}
\resizebox{\linewidth}{!}{$
k_3 > \frac{R_3^2\!\left[\psi_1 k_1k_2 - R_1^2R_2 k_2 \gamma_{13}
- R_1R_2^2\!\left(\gamma_{12}k_1 - R_1\right)\right]}{%
k_1k_2 \Bigl[\psi_1 R_3 \gamma_{23}
- R_{12}R_{13}R_{23}\gamma_1\gamma_2\gamma_3/4
- R_1 \gamma_{12} R_{23}^2 \gamma_3^2/4
- R_2 \gamma_{13} R_{13}^2 \gamma_2^2/4 \Bigr]
- k_1 R_2^2 \psi_2 - k_2 R_1^2 \psi_3 - R_1^2 R_2^2 R_3 \gamma_{23}}
$}
\end{equation}
\begin{proof}
Expanding the $8^{th}$-order principal minor, we have
\begin{align*}
A_8&=\left|\begin{smallmatrix}
        R_1 \gamma_{_{12}} & -R_1 & 0 & -R_{12}\gamma_{_{1}}/2 & 0 & 0 & -R_{13}\gamma_{_{2}}/2 & 0\\
-R_1 & k_1 & 0 & 0 & 0 & 0 & 0 & 0\\
0 & 0 & (1+\eta_1+\eta_2) & 0 & 0 & -\eta_1 & 0 & 0\\
- R_{12}\gamma_{_{1}}/2 & 0 & 0 & R_2 \gamma_{_{13}} & -R_2 & 0 & -R_{23} \gamma_{_{3}}/2 & 0\\
0 & 0 & 0 & -R_2 & k_2 & 0 & 0 & 0\\
0 & 0 & -\eta_1 & 0 & 0 &  (1+\eta_1+\eta_3) & 0 & 0 \\
-  R_{13}\gamma_{_{2}}/2 &  0 & 0 & - R_{23}\gamma_{_{3}}/2 & 0 & 0 & R_3 \gamma_{_{23}} & -R_3\\
0 & 0 & 0 & 0 & 0 & 0 & -R_3 & k_3
\end{smallmatrix}\right|\\
&=k_3A_7-R_3^2A_6 >0,
\end{align*}
which yields $ k_3> \frac{R_3^2A_6}{A_7}.$ Using the expansions for \(A_6\) and \(A_7\), we have \eqref{Eqn2.8}. Notice that \(A_9>0\) follows immediately from \eqref{Eqn2.8} as well
\begin{align*}
A_9&=\left|\begin{smallmatrix}
        R_1 \gamma_{_{12}} & -R_1 & 0 & -R_{12}\gamma_{_{1}}/2 & 0 & 0 & -R_{13}\gamma_{_{2}}/2 & 0 & 0\\
-R_1 & k_1 & 0 & 0 & 0 & 0 & 0 & 0 & 0\\
0 & 0 & (1+\eta_1+\eta_2) & 0 & 0 & -\eta_1 & 0 & 0 & -\eta_2\\
- R_{12}\gamma_{_{1}}/2 & 0 & 0 & R_2 \gamma_{_{13}} & -R_2 & 0 & -R_{23} \gamma_{_{3}}/2 & 0 & 0\\
0 & 0 & 0 & -R_2 & k_2 & 0 & 0 & 0 & 0\\
0 & 0 & -\eta_1 & 0 & 0 &  (1+\eta_1+\eta_3) & 0 & 0 & -\eta_3\\
-  R_{13}\gamma_{_{2}}/2 &  0 & 0 & - R_{23}\gamma_{_{3}}/2 & 0 & 0 & R_3 \gamma_{_{23}} & -R_3 & 0\\
0 & 0 & 0 & 0 & 0 & 0 & -R_3 & k_3 & 0\\
0 & 0 & -\eta_2 & 0 & 0 & -\eta_3 & 0 & 0 & (1+\eta_2+\eta_3)
\end{smallmatrix}\right|\\
&= \frac{1+2(\eta_1+\eta_2+\eta_3)+3(\eta_1\eta_2+ \eta_1\eta_3 +\eta_2\eta_3)}{(1+\eta_1+\eta_2)(1+\eta_1+\eta_3)-\eta_1^2} A_8.
\end{align*}
Hence, the inequality \(A_9>0\) does not produce any new bound on the gains.
\end{proof}
\end{lemma}
Now, we are ready to present the main result on the stability of \eqref{Eqn2.1} in the following theorem.

\begin{theorem}\label{thm2.1} There exists lower bound on the state feedback gains $k_{1,2,3}$ so that the triangularly coupled loop thermosyphon system \eqref{Eqn2.1} with the single-state feedback $u_i=-k_iy_i$, $1\leq i\leq 3$, is globally asymptotically stable at the origin.
\begin{proof}
Consider the Lyapunov function 
\[
V(x)=\tfrac{1}{2}\sum_{i=1}^3 \Big(\tfrac{R_i}{p}x_i^2+y_i^2+z_i^2\Big),
\]
which is positive definite and radially unbounded. The time derivative of $V$ along the state trajectories can be written as
\begin{align*}
    \dot V(x)= -\,x^\top Ax,
\end{align*}
where $A$ is the symmetric matrix defined in (\ref{mat1}). We define
\begin{align}\label{EqnK_i's}
\begin{cases}
    &\mathcal{K}_{1} := \textit{ maximum lower bound between }\eqref{Eqn2.4} \textit{ and }  \eqref{Eqn2.6}\\
    &\mathcal{K}_{2} := \textit{ maximum lower bound between }\eqref{Eqn2.5} \textit{ and }  \eqref{Eqn2.7}.
 \end{cases}   
\end{align}
Then, if $k_1 > \mathcal{K}_{1}$ and $k_2 > \mathcal{K}_{2}$, all leading principal minors of $A$ up to $A_7$ are positive according to Lemma \ref{combined_lemma1&2} to Lemma \ref{lemma2.4}. Lastly, if $k_3 \text{ satisfies } (\ref{Eqn2.8})$, principal minors $A_8$ and $A_9$ are also positive based on Lemma \ref{lemma2.5}. Therefore, there exists lower bounds on $k_{1,2,3}$ so that all the principal minors of $A$ are positive. Then, by Sylvester’s criterion, $A$ is positive definite, and thus $\dot V<0$.
\end{proof}
\end{theorem}
Theorem \ref{thm2.1} highlights an important physical insight: the stability conditions for the coupled thermosyphon system depend only on the Rayleigh numbers $R_i$ and the momentum coupling parameter $\gamma_i$, while remaining unaffected by the thermal coupling parameter $\eta_i$. In practical terms, this indicates that the degree of thermal coupling does not have any bearing on the stability of the triangularly coupled loop system \eqref{Eqn2.1}.

From \eqref{Eqn2.6}, we observe that the lower bound for $k_1$ depends explicitly on all the Rayleigh numbers $R_i's$, the momentum coupling parameters $\gamma_i's$, and the auxiliary quantities $\psi_i's$ defined in \eqref{Eqn2.2}. Interestingly, the geometric symmetry of the triple-loop system also creates an algebraic symmetry in the lower bound in \eqref{Eqn2.6}: if we fix the first loop and interchange the indices of the remaining loops, i.e., swap $R_2 \leftrightarrow R_3$ and $\gamma_1 \leftrightarrow \gamma_2$, the lower bound is invariant. This invariance arises from the symmetric forms of $\gamma_{ij}=1+\gamma_i+\gamma_j$ and $R_{ij}=R_i+R_j$, which preserve the algebraic form of both the numerator and denominator in \eqref{Eqn2.6}. Consequently, the constraint on $k_1$ is unaffected by interchanging the second and third loops. Such a symmetry in the stability bounds does not exist in those bounds derived from the two-loop system in \cite{ref14}. This is because the two-loop system does not have the closed-loop topological structure as in this work. The bounds obtained from the triangularly coupled triple loop system are also much more complicated than its two-loop counterpart due to differences in the coupling mechanism and dimensionality.

In essence, Theorem \ref{thm2.1} established the controllability of \eqref{Eqn2.1} via local state feedback that is incorporated into the $y$-equations of \eqref{Eqn2.1}. The numerical simulations also verified the convergence of the states toward the origin. Since the simulation results are quite close to those from the case of adaptive control, which will be discussed later, we instead present some numerical results on the tracking ability of the states with respect to non-stationary reference signals. In this simulation, we specify the following key parameter values as the Rayleigh numbers $R_1=52, R_2=50, R_3=56,$ the momentum coupling parameters $\gamma_1=0.4, \gamma_2=0.5, \gamma_3=0.7$ and the thermal coupling parameters $\eta_1=0.2, \eta_2=0.5, \eta_3=0.2$. The feedback gain values are $k_1=50$, $k_2=48$, and $k_3=52$. In the experiment setup, we make the $y_2$-state track a reference signal in the form of a mixed sinusoids, $i.e.$  $y_{r_2}=15\sin(2t)+12\cos(3t)$ along with $y_{r_1}=y_{r_3}=0$ in \eqref{Eqn1.1}. It is shown in Figure \ref{fig_tr} that $y_2$ quickly captures the waveform of the reference signal and stays close in tracking the input signal. The state trajectories of $y_1$ and $y_3$ in their steady state also exhibit some characteristics of a sinusoid due to the thermal and momentum coupling terms in the equations reacting to the steady state response of $y_2$. This in turn demonstrates the coupling effect on the dynamics of the neighboring loops.

\begin{figure}[H]
\centering
\includegraphics[width=8.0cm]{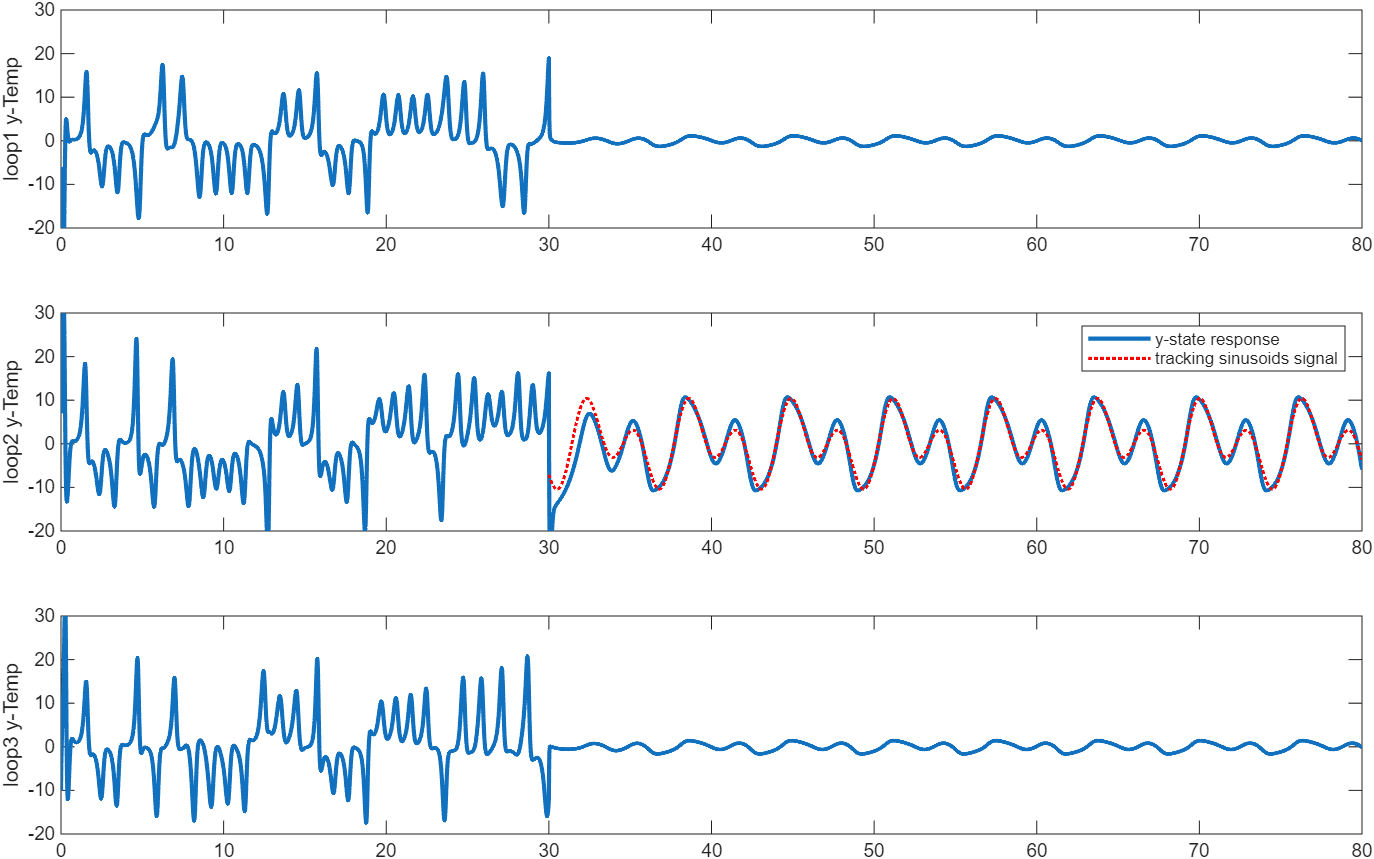}
\caption{Tracking error of the $y_2$-state with respect to a sinusoid reference signal\label{fig_tr}}
\end{figure}

The stability bounds on the feedback gains established so far rely on the assumption that the system parameters $\{R_i,\gamma_i\}_{i=1}^3$ are known a \textit{priori}. In practice, however, these parameters are usually unavailable. To this end, we turn to a dynamic gain search scheme, in which the feedback gains evolve dynamically rather than being fixed. More specifically, each $k_i$ is treated as an extended state governed by its own dynamic equation. In this design, the gains monotonically increase (from zero) when the system is unstable till the gains surpass their respective stability bounds.  As a result of Theorem \ref{thm2.1}, the  trajectories converge asymptotically to the origin.  Based on this idea, we propose the following decentralized adaptive control system for the triangularly coupled loop thermosyphon when the system parameters are unknown:

\begin{equation}\label{Eqn2.9}
\begin{cases}
& \dot{x}_1=p\left\{\left(y_1-x_1\right)-\gamma_1\left(x_1-x_2\right)-\gamma_2\left(x_1-x_3\right)\right\} \\
& \dot{y}_1=R_1 x_1 -x_1 z_1 -k_1y_1\\
& \dot{z}_1=x_1 y_1-z_1-\eta_1\left(z_1-z_2\right)-\eta_2\left(z_1-z_3\right) \\
& \dot{x}_2=p\left\{\left(y_2-x_2\right)-\gamma_1\left(x_2-x_1\right)-\gamma_3\left(x_2-x_3\right)\right\} \\
& \dot{y}_2=R_2 x_2 -x_2 z_2 -k_2y_2\\
& \dot{z}_2=x_2 y_2-z_2-\eta_1\left(z_2-z_1\right)-\eta_3\left(z_2-z_3\right) \\
& \dot{x}_3=p\left\{\left(y_3-x_3\right)-\gamma_2\left(x_3-x_1\right)-\gamma_3\left(x_3-x_2\right)\right\} \\
& \dot{y}_3=R_3 x_3 -x_3 z_3 -k_3y_3\\
& \dot{z}_3=x_3 y_3-z_3-\eta_2\left(z_3-z_1\right)-\eta_3\left(z_3-z_2\right)\\
& \dot{k}_i=\alpha_iy_i^2, \hspace{4cm} \ i = 1,2,3,
\end{cases}
\end{equation}
where $\alpha_i$ is the learning rate constant. Note that, the closed-loop system \eqref{Eqn2.9} evolves in the extended space space $\mathbb{R}^{9} \times \mathbb{R}_{+}^3$, where $\mathbb{R}_{+}^3$ accounts for the positive adaptive gains $k_i$.

It follows from the system equations \eqref{Eqn2.9} that the evolution of $k_i$ is strictly monotonically increasing whenever $y_i \neq 0$. Once the adaptive gains exceed the thresholds specified in Theorem \ref{thm2.1}, the trajectories are driven to the origin asymptotically. In other words, as $y_i \to 0$, each gain $k_i$ converges to a steady state $k_i^*$ that satisfies the lower bound conditions in Theorem \ref{thm2.1}. A natural initialization is $k_i(0)=0$. To simplify the stability analysis, we introduce change of variables in $k_i = \hat{k}_i + k_i^*$. Substituting this expression into \eqref{Eqn2.9}, we obtain an equivalent system
\begin{equation}\label{Eqn2.10}
\begin{cases}
& \dot{x}_1=p\left\{\left(y_1-x_1\right)-\gamma_1\left(x_1-x_2\right)-\gamma_2\left(x_1-x_3\right)\right\} \\
& \dot{y}_1=R_1 x_1 -x_1 z_1 -\hat{k}_1y_1 -{k}^{*}_1y_1\\
& \dot{z}_1=x_1 y_1-z_1-\eta_1\left(z_1-z_2\right)-\eta_2\left(z_1-z_3\right) \\
& \dot{x}_2=p\left\{\left(y_2-x_2\right)-\gamma_1\left(x_2-x_1\right)-\gamma_3\left(x_2-x_3\right)\right\} \\
& \dot{y}_2=R_2 x_2 -x_2 z_2 -\hat{k}_2 y_2 -{k}^{*}_2y_2\\
& \dot{z}_2=x_2 y_2-z_2-\eta_1\left(z_2-z_1\right)-\eta_3\left(z_2-z_3\right) \\
& \dot{x}_3=p\left\{\left(y_3-x_3\right)-\gamma_2\left(x_3-x_1\right)-\gamma_3\left(x_3-x_2\right)\right\} \\
& \dot{y}_3=R_3 x_3 -x_3 z_3 -\hat{k}_3y_3 -{k}^{*}_3y_3\\
& \dot{z}_3=x_3 y_3-z_3-\eta_2\left(z_3-z_1\right)-\eta_3\left(z_3-z_2\right)\\
& \hat{\dot{k}}_i=\alpha_iy_i^2, \hspace{4.5cm} i = 1,2,3.
\end{cases}
\end{equation}
Hence, this transformation places the origin of the extended state space as the equilibrium of \eqref{Eqn2.10}. We state the following theorem concerning the stability of the adaptive control system \eqref{Eqn2.10}.

\begin{theorem}\label{thm2.2} The augmented triangularly coupled closed-loop system \eqref{Eqn2.10}, in which $k_1^*> \mathcal K_1,\quad k_2^*> \mathcal K_2,$ and $k_3^*$ satisfies \eqref{Eqn2.8}, with the adaptive law $\dot{\hat{k}}_i=\alpha_i y_i^2$, $1\le i\le 3$, is globally asymptotically stable at the origin.
\begin{proof}
We use the Lyapunov function
\begin{align*}
    V(x,\hat{k}) = \frac{1}{2}\sum_{i=1}^3 \left( \frac{R_i}{p}x_i^2 + y_i^2 + z_i^2 + \frac{1}{\alpha_i}\hat{k}_i^2 \right)>0.
\end{align*}
The time derivative of $V$ along the state trajectories of \eqref{Eqn2.10} is given by
\begin{align*}
    \dot{V} &= \sum_{i=1}^3 \Big( \tfrac{R_i}{p}x_i \dot{x}_i + y_i \dot{y}_i + z_i \dot{z}_i + \tfrac{1}{\alpha_i}\hat{k}_i \dot{\hat{k}}_i \Big) \\
           &= -\,x^\top A(k^*)x,
\end{align*}
where $k^*=[k_1^*,\quad k_2^*, \quad k_3^*]^\top$ and $A(k^*)$ is the matrix in \eqref{mat1} evaluated at $k^*$. Since $k_1^* > \mathcal K_1$, $k_2^*> \mathcal K_2$, and $k_3^*$ satisfies \eqref{Eqn2.8}, $A(k^*)$ is positive definite. It follows that $\dot V<0$. Since $V$ is radially unbounded, system \eqref{Eqn2.10} is globally asymptotically stable at the origin.
\end{proof}
\end{theorem}
We believe that the convergence of the adaptive gain is guaranteed by Theorem \ref{thm2.1}, where it is established that there exists lower bounds on the feedback gains so that the system is stabilized at the origin when the gains are greater than their respective bounds. The dynamic equations of the adaptive gains also guarantee that the gains are monotonically increasing whenever the $y$-state is nonzero. As soon as the bounds surpass the threshold values, the state trajectories asymptotically approach the origin. According to the governing equation of the gain \eqref{Eqn2.10}, the time derivative of the gain approaches zero. As a result, the gain converges to a set value. The range of the learning rate closely relates to the above analysis. The learning rate parameter appears in the Lyapunov function. However, it does not play a role in the stability analysis as it is being canceled from the time derivative of the Lyapunov function along the state trajectory. Yet, it indeed controls the transient response because the parameter dictates the rate of change of the adaptive gain. We will demonstrate this through the numerical simulations as follows.

The adaptive gain search via the dynamic equations on the feedback gains is of interest in practice. This is because, with unknown system parameters, the lower bounds on the gains cannot be computed beforehand. We observe the performance of such an adaptive controller design through numerical simulations. In the following example, the key system parameter values are $R_1=35, R_2=45, R_3=38,$ the momentum coupling parameters $\gamma_1=0.1, \gamma_2=0.3, \gamma_3=0.2$ and the thermal coupling parameters $\eta_1=0.1, \eta_2=0.1, \eta_3=0.2$. For the simulation results shown in  Figures \ref{fig2.1}, \ref{fig2.2}, and \ref{fig2.3}, the initial state values are given in $[-8, -6, 5, 3, 7, 11, 10, -10, 2]$; the initial value of each gain is set at $0$. As shown in Figure \ref{fig2.1}(b) and \ref{fig2.2}(b), the gains monotonically increase and asymptotically approach their threshold values to stabilize the system, which is evidenced by the system's time response of fluid velocity shown in Figure \ref{fig2.1}(a) and \ref{fig2.2}(a), as well as the time response of the temperature state in $y$ in Figure \ref{fig2.3}(a) and the temperature state in $z$ in \ref{fig2.3}(b). The learning rate for Figure \ref{fig2.1} is $0.8$, whereas it is set at $2.5$ for Figure \ref{fig2.2} and \ref{fig2.3}. As a result, the transient response shown in Figure \ref{fig2.2} and \ref{fig2.3} is much shorter than that of Figure \ref{fig2.1}, which agrees with the analysis above. The steady state errors for all 9 states of the numerical simulation of \eqref{Eqn2.10} are shown in Table \ref{tab:steaderr1}. Those data are collected 30 seconds after the activation of the adaptive controllers. The asymptotic convergence of the states to the equilibrium is evident from the numerical values. This result also agrees with the qualitative analysis presented above.
\begin{figure}[H]
\centering
  \centering
\includegraphics[scale=0.48]{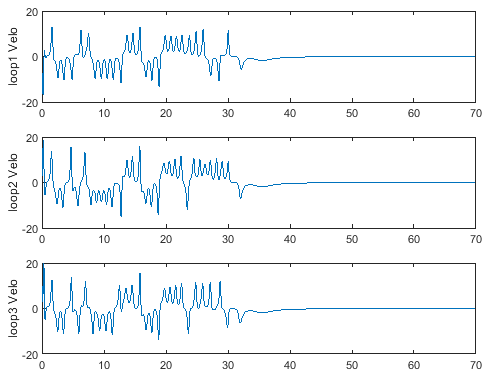}
  \centering
\includegraphics[scale=0.48]{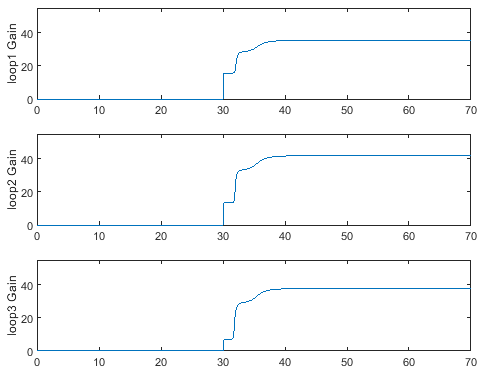}
\caption{(\textbf{a}) Stabilized fluid flow in each loop with controller activated at $t=30s;$ (\textbf{b}) dynamic gain search with a learning rate at $0.8$.\label{fig2.1}}
\end{figure}

\begin{figure}[H]
\centering
  \centering
\includegraphics[scale=0.48]{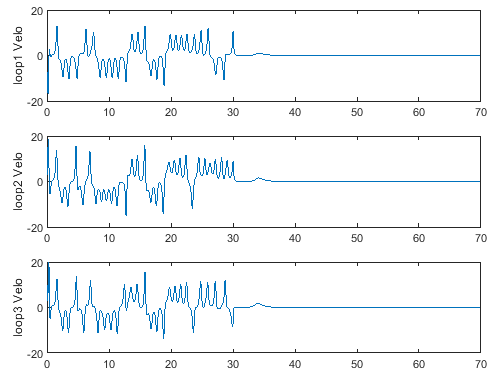}
  \centering
\includegraphics[scale=0.48]{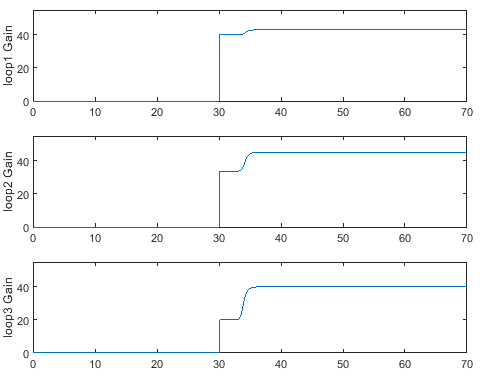}
\caption{(\textbf{a}) Stabilized fluid flow in each loop with controller activated at $t=30s;$ (\textbf{b}) dynamic gain search with a learning rate at $2.5$.\label{fig2.2}}
\end{figure}

\begin{figure}[H]
\centering
  \centering
\includegraphics[scale=0.27]{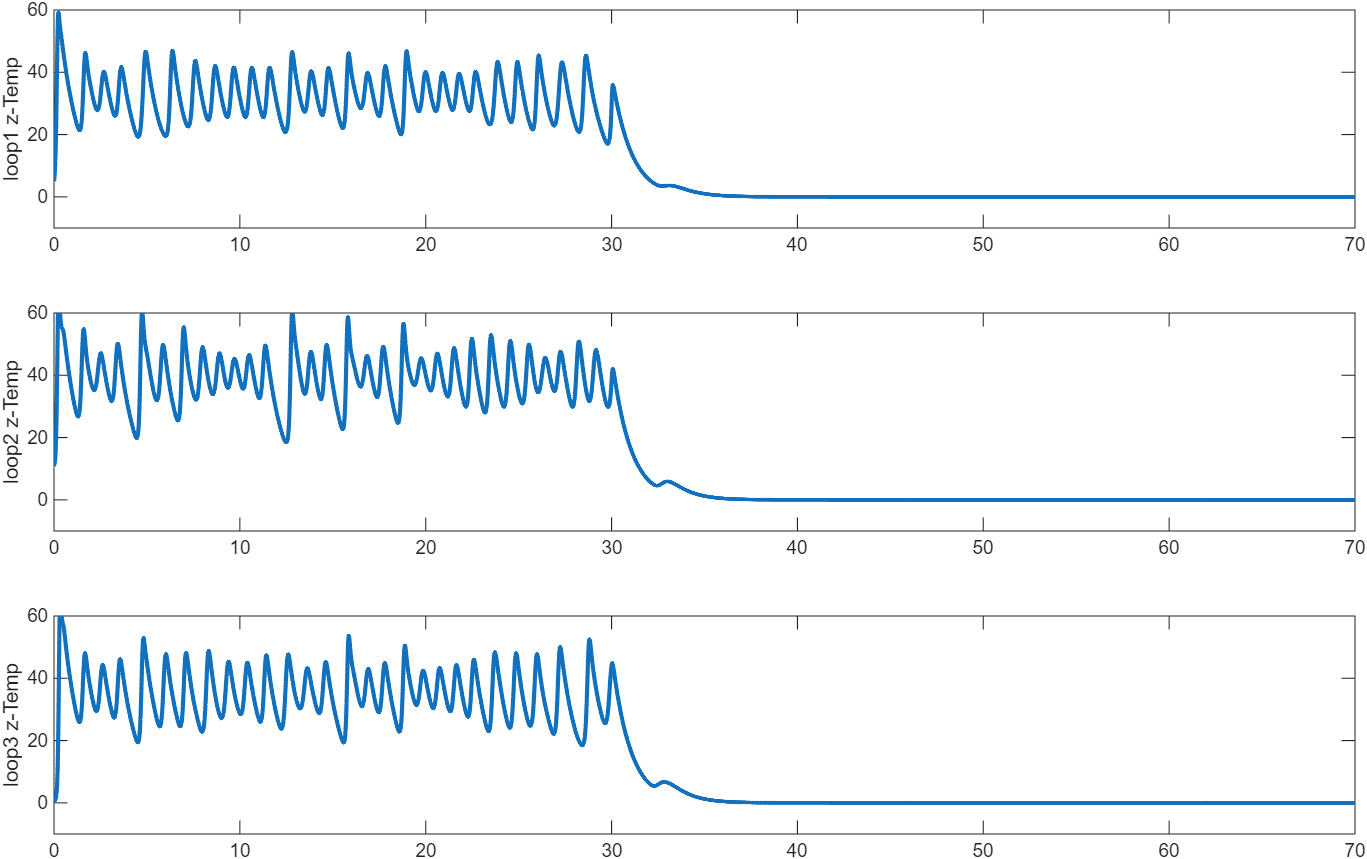}
  \centering
\includegraphics[scale=0.27]{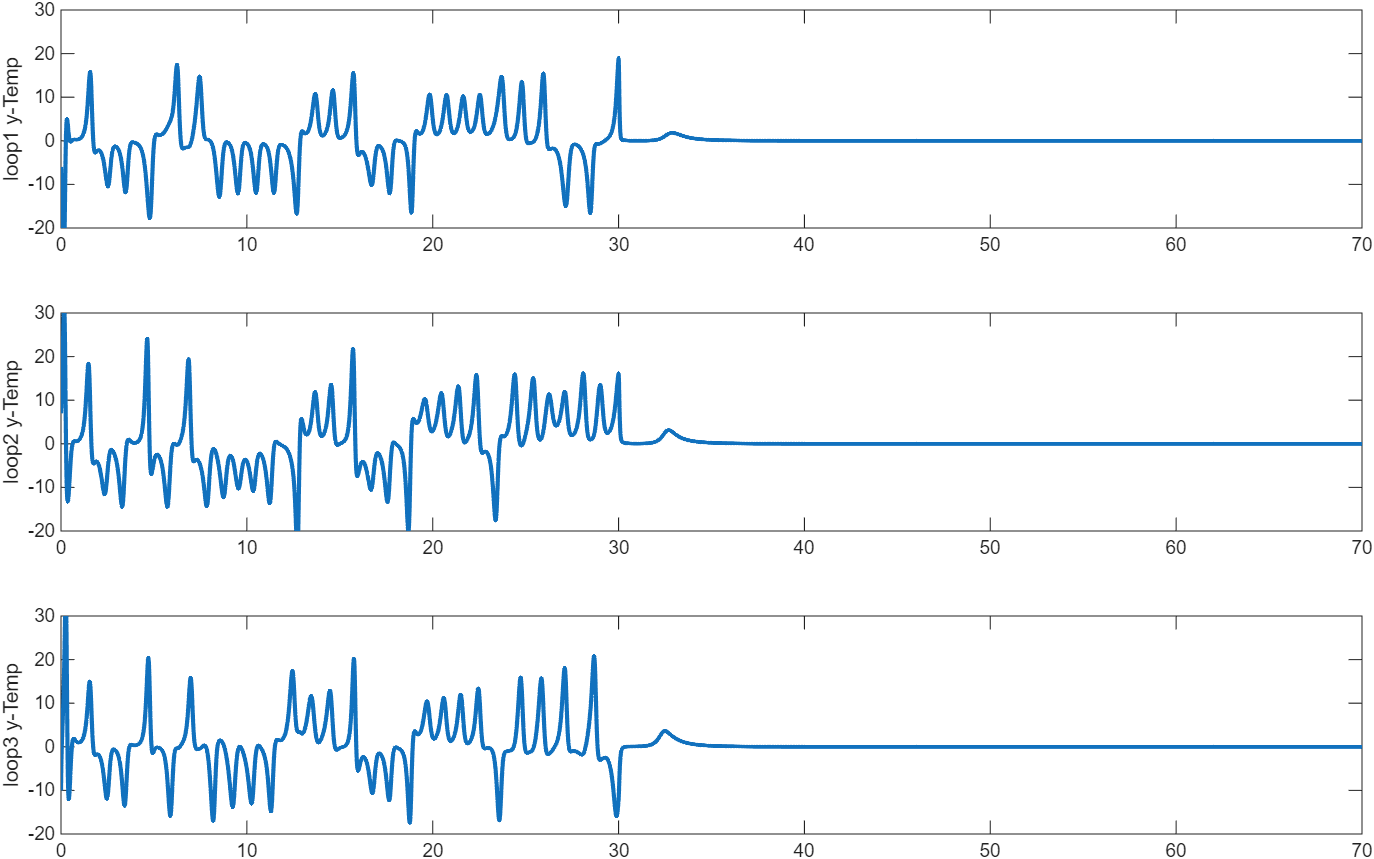}
\caption{(\textbf{a}) Stabilized $y$- temperature state in each loop (\textbf{b}) Stabilized $z$- temperature state in each loop \label{fig2.3}}
\end{figure}

\begin{table}[H]
\caption{Steady state error of the adaptive control system }
\centering
\scriptsize                          
\setlength{\tabcolsep}{3pt}          
\resizebox{\linewidth}{!}{%
\begin{tabular}{|*{9}{c|}}
\hline
$x_1$ & $y_1$ & $z_1$ & $x_2$ & $y_2$ & $z_2$ & $x_3$ & $y_3$ & $z_3$ \\
\hline
     8.9395e-13 & 8.7501e-13 &  1.2399e-16 & 9.0648e-13 &  8.9291e-13 & 1.2677e-16 &  9.0067e-13 & 8.8380e-13 & 1.2559e-16\\
  8.9232e-13 & 8.7342e-13 & 1.2354e-16 & 9.0483e-13 & 8.9128e-13 & 1.2631e-16 &   8.9902e-13 & 8.8219e-13 &  1.2513e-16\\
  8.9069e-13 & 8.7182e-13 & 1.2309e-16 & 9.0318e-13 & 8.8966e-13 & 1.2585e-16 &
    8.9738e-13 & 8.8057e-13  & 1.2468e-16\\
   8.8906e-13 & 8.7023e-13 &  1.2264e-16 & 9.0153e-13 & 8.8803e-13 & 1.2539e-16 &
     8.9574e-13 & 8.7897e-13 &  1.2422e-16\\
   8.8744e-13 & 8.6864e-13 &  1.2219e-16 & 8.9988e-13 & 8.8641e-13 & 1.2493e-16 &
     8.9410e-13 & 8.7736e-13  & 1.2377e-16\\
   8.8582e-13 & 8.6705e-13 & 1.2174e-16 & 8.9824e-13 & 8.8479e-13 & 1.2448e-16 &
  8.9247e-13 & 8.7576e-13 &  1.2332e-16\\
  8.8420e-13 & 8.6547e-13 & 1.2130e-16 & 8.9660e-13 & 8.8317e-13 & 1.2402e-16 &   8.9084e-13 & 8.7416e-13 & 1.2287e-16\\
  8.8259e-13 & 8.6389e-13 & 1.2086e-16 & 8.9496e-13 & 8.8156e-13 & 1.2357e-16 &
8.8921e-13 & 8.7256e-13 & 1.2242e-16\\
  8.8097e-13 & 8.6231e-13 & 1.2042e-16 & 8.9333e-13 & 8.7995e-13 & 1.2312e-16 &
    8.8759e-13 & 8.7097e-13 & 1.2197e-16\\
  8.7937e-13 & 8.6074e-13 & 1.1998e-16 & 8.9169e-13 & 8.7834e-13 & 1.2267e-16 &
  8.8597e-13 & 8.6938e-13 & 1.2153e-16\\
\hline
\end{tabular}%
\label{tab:steaderr1}}
\end{table}

\section{System with Uncertainties and Disturbance Rejection Control}\label{sec3}
The system of differential equations \eqref{Eqn2.1} is known as an idealization of the fluid flows in a triangularly coupled loop thermosyphon system because it only provides an approximation of the dynamics.  During the mathematical modeling process, higher order terms from the Fourier series of the temperature function are dropped to simplify the equations. Also, the uncertainties from the high heat environment are not considered in the differential equations.  All these unmodeled components contribute to the small discrepancies between a good model and the real system.  The question we want to address here is whether we can effectively stabilize the system if we expand mathematical modeling by including the unmodeled components as an unknown function in the dynamic equations. To this end, we modify the original system equations \eqref{Eqn2.1} by incorporating an unknown function that represents the disturbance, a universal term for uncertainties and unmodeled components of the system, as follows
\begin{equation}\label{Eqn3.1}
\begin{cases}
& \dot{x}_1=p\left\{\left(y_1-x_1\right)-\gamma_1\left(x_1-x_2\right)-\gamma_2\left(x_1-x_3\right)\right\} \\
& \dot{y}_1=R_1 x_1-k_1y_1-x_1 z_1 +f_1(y_1,\zeta_1(x_1,z_1,t))+u_1\\
& \dot{z}_1=x_1 y_1-z_1-\eta_1\left(z_1-z_2\right)-\eta_2\left(z_1-z_3\right) \\
& \dot{x}_2=p\left\{\left(y_2-x_2\right)-\gamma_1\left(x_2-x_1\right)-\gamma_3\left(x_2-x_3\right)\right\} \\
& \dot{y}_2=R_2 x_2-k_2y_2-x_2 z_2 +f_2(y_2,\zeta_2(x_2,z_2,t))+u_2\\
& \dot{z}_2=x_2 y_2-z_2-\eta_1\left(z_2-z_1\right)-\eta_3\left(z_2-z_3\right) \\
& \dot{x}_3=p\left\{\left(y_3-x_3\right)-\gamma_2\left(x_3-x_1\right)-\gamma_3\left(x_3-x_2\right)\right\} \\
& \dot{y}_3=R_3 x_3-k_3y_3-x_3 z_3 +f_3(y_3,\zeta_3(x_3,z_3,t))+u_3\\
& \dot{z}_3=x_3 y_3-z_3-\eta_2\left(z_3-z_1\right)-\eta_3\left(z_3-z_2\right)
\end{cases},
\end{equation}
where the unknown functions $f_i$ and $\zeta_i$, $i=1,2,3,$ represent the disturbances, uncertainties, and unmodeled components due to simplification and idealization in the mathematical modeling. We assume they only depend on the local states.  Each of the controllers $u_i$ consists of an approximation $f_i$, denoted by $\hat{f_i}$, and a proportional controller as detailed in Section \ref{sec2}. More specifically,
\begin{align}\label{Eqn3.2}
    u_i=-\hat{f_i}-k_iy_i ,\quad i=1,2,3.
\end{align}
The approximation $\hat{f_i}$ is the output of the extended Luenberger state observer, known as the extended state observer (ESO) in the regime of active disturbance rejection control (ADRC).  The concept of ADRC was first introduced by Han in his seminal paper \cite{ref16} to address the challenges of robust control when there exist significant uncertainties in dynamics and external disturbances \cite{ref17}.  The mechanism of ADRC was systematically and elegantly disseminated by Gao in \cite{ref18}.  Gao also showed that parameterized linear gains make ADRC much more practical to implement in his paper \cite{ref19}. Returning to the controller \eqref{Eqn3.2} in conjunction with \eqref{Eqn3.1}, it is easy to see that the performance of the ADRC controller is dependent on the error between the ESO output and the disturbance, $i.e. |f_i-\hat{f_i}|$. We have shown that \eqref{Eqn3.1} is globally asymptotically stable at equilibrium in section \ref{sec2} without disturbances $f_i$. It is crucial for us to establish the convergence of the extended state observer before we could show that \eqref{Eqn3.1} is stable with the existence of uncertainties in the system.  To this end, we consider the following $n^{th}$ order nonlinear plant to illustrate the ESO design and explore its convergence properties
\begin{align}\label{Eqn3.3}
    y^{(n)}=f(t,y,y',\cdots, y^{(n-1)},\zeta(t))+ u,
\end{align}
where $\zeta(t)$ represents the conglomerate of the external states, uncertainties including disturbances and / or unmodeled components, and $u$ is the controller. The corresponding state-space representation of \eqref{Eqn3.3} is
\begin{align}\label{Eqn3.4}
    \dot{x}_1=x_2, \quad \dot{x}_2=x_3, \cdots, \quad \dot{x}_{n-1}=x_n, \quad \dot{x}_n=x_{n+1}+u, \quad \dot{x}_{n+1}=g(x,\zeta),
\end{align}
where $x_1=y, g=df/dt$, and $x_{n+1}=f$ is known as the extended state or pseudo-state.  One of the outputs from the extended-state observer yields an approximation of this extended state $x_{n+1}$, $i.e.$ the $\hat{f_i}$ in \eqref{Eqn3.2}. The state vector $x$ consists of all the state variables in \eqref{Eqn3.4}.  Any state observer of \eqref{Eqn3.4}, such as the Luenberger state observer, will yield approximation of the derivatives of y and the nonlinear plant $f$ because $f$ is treated as an extended state.  The ESO of \eqref{Eqn3.4} is given by
\begin{align}\label{Eqn3.5}
\begin{cases}
& \dot{\hat{x}}_1=\hat{x}_2+ \beta_1 (x_1-\hat{x}_1) \\
& \dot{\hat{x}}_2=\hat{x}_3+ \beta_2 (x_1-\hat{x}_1) \\
& \cdots\\
& \dot{\hat{x}}_{n-1}=\hat{x}_n+ \beta_{n-1} (x_1-\hat{x}_1) \\
& \dot{\hat{x}}_{n}=\hat{x}_{n+1}+ \beta_{n} (x_1-\hat{x}_1) + u \\
& \dot{\hat{x}}_{n+1}= \beta_{n+1} (x_1-\hat{x}_1) + g(\hat{x}, \zeta),
\end{cases}
\end{align}
where $\hat{x}$ is the state vector of \eqref{Eqn3.5}, and the $\beta_i's$  are the observer gains. They are chosen so that the characteristic polynomial of \eqref{Eqn3.5} is Hurwitz.  A convenient option is to define the observer gains as
\begin{align}\label{Eqn3.6}
    \beta_i= \binom{n+1}{i}B^{i}, \quad i=1,2, \cdots, (n+1),
\end{align}
in which $B$ is the bandwidth of the observer, so that the characteristic polynomial of \eqref{Eqn3.5} turns out to be $p(s)=(s+B)^{n+1}$. Hence, it is Hurwitz because $B>0$. It will be established that $\|x-\hat{x} \|\to 0$, $i.e.$ the observer states and the plant states are locked in together asymptotically. It is sufficient to examine the error system between \eqref{Eqn3.4} and \eqref{Eqn3.5} , which is given by
\begin{align}\label{Eqn3.7}
\begin{cases}
&\dot{e}_1=e_2-\beta_1 e_1\\
&\dot{e}_2=e_3-\beta_2 e_1\\
& \cdots\\
&\dot{e}_{n-1}=e_n-\beta_{n-1} e_1\\
&\dot{e}_{n}=e_{n+1}-\beta_{n} e_1\\
& \dot{\hat{x}}_{n+1}= -\beta_{n+1} e_1 + g(x, \zeta)- g(\hat{x}, \zeta),
\end{cases}
\end{align}
where $e_i=x_i-{\hat{x}_i}$, the error state. Let $e=\left[e_1 \quad e_2 \quad \cdots \quad e_n \quad e_{n+1}\right]^\top$, then the following theorem establishes the asymptotic stability of \eqref{Eqn3.7}, see \cite{ref20} for details of the proof.

\begin{theorem}\label{thm3.1}
Suppose that $g(x,\zeta)$ is Lipschitz with respect to $x$. Then, there exists a bandwidth $B\geq B_0 >0$ such that the system \eqref{Eqn3.7} is asymptotically stable.
\end{theorem}

Theorem \ref{thm3.1} implies that $\lim_{t \to \infty} \|e(t)\|=0$ if the disturbance function satisfies some smoothness condition over the state vector provided that the observer bandwidth is large enough.  Now, in light of \eqref{Eqn3.5}, we design an ESO for each of the $y-$equation where the disturbance resides in as follows
\begin{align}\label{Eqn3.8}
\begin{cases}
&\dot{\hat{y}}_{i_1}= \hat{y}_{i_2} +\beta_{i_{1}} (y_{i_1}- \hat{y}_{i_1}) + u_i\\
&\dot{\hat{y}}_{i_2}=  \beta_{i_{2}} (y_{i_1}- \hat{y}_{i_1}) + g_i (\hat{y}_i, \zeta_i) \hspace{2cm} \textit{for } i=1,2,3
\end{cases}
\end{align}
where $\hat{y}_{i_{1}}$ is the ESO output for $y_i$ , and $\hat{y}_{i_2}$  is the approximation of the right hand side of the  $y_i$-equation in \eqref{Eqn3.1}, $i.e.$ let $ F_i=R_ix_i-y_i-x_iz_i+f_i(y_i,\zeta_i)$. Suppose the disturbance function $f_i$ has continuous partial derivative with respect to $y_i$, it is easy to see that $\dfrac{\partial{F_i}}{\partial{y_i}}$ is also continuous. Therefore, $F_i$  is Lipschitz with respect to $y_i$.  According to Theorem \ref{thm3.1}, by designating appropriate values, such as \eqref{Eqn3.6} where $n=1,$ to the observer parameters $\beta{_{{i}_{1,2}}}$ in \eqref{Eqn3.8}, we have $|\hat{y}_{i_2}-F_i| \to 0$. Let $\hat{f_i}(\hat{y}_i,\zeta_i)=\hat{y}_{{i}_{2}}-R_ix_i+y_i+x_iz_i$, then we conclude that
\begin{align}\label{Eqn3.9}
    |\hat{f_i}(\hat{y}_i,\zeta_i) - f_i({y}_i, \zeta_i)| \to 0.
\end{align}
The following theorem concerns the stability of the ADRC-controlled system \eqref{Eqn2.1}.

\begin{theorem}\label{thm3.2}
    The closed loop triangularly coupled thermosyphon system \eqref{Eqn3.1} with disturbances $f_i(y_i,\zeta_i)$, and the controllers $u_i=\hat{f}_i -k_iy_i$, $i=1,2,3,$ where $\hat{f}_i$ is the approximation of the disturbance constructed from the output of the extended state observer \eqref{Eqn3.8}, is globally asymptotically stable at the origin if the following conditions are satisfied
\begin{itemize}
    \item [(i)] $\partial f_i/ \partial y_i, i=1,2,3,$ is continuous
    \item[(ii)]  the observer parameters $\beta_{i_{1}}$ and $\beta_{i_{2}}, i=1,2,3$, satisfy \eqref{Eqn3.6}, where $n=1$ 
    \item[(iii)] state feedback gains $k_{1,2}$ satisfy \eqref{EqnK_i's} and $k_3$ satisfy \eqref{Eqn2.8}.
\end{itemize}
\begin{proof} Consider the Lyapunov function associated with \eqref{Eqn3.1} as follows
   \begin{align*}
       V=\frac{1}{2} \sum_{i=1}^3 \left( \frac{R_i}{p}x_i^2+y_i^2+z_i^2 \right).
   \end{align*} 
The time derivative of $V$ along the state trajectory is written as
\begin{align}\label{Eqn3.10}
    \dot{V} =-x^\top Ax +\sum_{i=1}^3 y_ie_i,
\end{align}
where the state matrix A is given by \eqref{mat1}, $x=[x_1,\quad y_1, \quad z_1, \quad x_2, \quad  y_2, \quad z_2, \quad x_3,\quad y_3,\quad z_3]^\top$, and $e_i=f_i-\hat{f}_i$, the error of the disturbance approximation, after plugging $u_i=-\hat{f}_i-k_iy_i$ in \eqref{Eqn3.1}. Using the Cauchy-Schwarz inequality, and let $\bar{y}=[y_1, \quad y_2,\quad y_3]^\top, \bar{e}=[e_1, \quad e_2, \quad e_3]^\top,$ from \eqref{Eqn3.10} , we have
\begin{align}\label{Eqn3.11}
    \dot{V} \leq -x^\top Ax +\|\bar{y}\| \|\bar{e}\| \leq  -x^\top Ax +\|\bar{x}\| \|\bar{e}\|.
\end{align}
Because the feedback gains $k_i$ satisfy \eqref{Eqn2.8} and \eqref{EqnK_i's}, the state matrix $A$ is positive definite. Therefore, $ -x^\top Ax \leq -\lambda_{\min} \|x\|^2,$ where $\lambda_{\min}$ is the smallest eigenvalue, which satisfies $\lambda_{\min}>0$. It is easy to see that system \eqref{Eqn3.1} is dissipative because the divergence of the
velocity field is negative. Hence, the state trajectories are restricted in bounded region. Hence, there exists $M>0$, such that $0 < \|x\| < M$. With the conditions (i) and (ii) being
satisfied, according to Theorem \ref{thm3.1} and as a result of \eqref{Eqn3.9}, we have $\lim_{t\to \infty} e_i= 0, \quad i=1,2,3$. This implies that, $\forall \varepsilon>0, \exists \quad T_i>0,$ such that $|e_i|< \lambda_{\min}\|x\|/(2\sqrt{3)}$. Let $T^*=\max \{T_1, T_2, T_3\},$ We then have $\|e\| < \lambda_{\min} \|x\|/2$ whenever $t>T^*$. This allows us to rewrite \eqref{Eqn3.11} to get
\begin{align*}
    \dot{V}< -\lambda_{\min}\|x\|^2+\frac{1}{2} \lambda_{\min} \|x\|^2 = -\frac{1}{2} \lambda_{\min} \|x\|^2 < 0.
\end{align*}
Therefore, the closed loop system \eqref{Eqn3.1} with a proportional state feedback control and a disturbance rejection control is globally asymptotically stable at the origin.
\end{proof}
\end{theorem}
It is readily seen from \eqref{Eqn3.11} that the convergence of the ESO plays an essential role in
the stability analysis of \eqref{Eqn3.1}. The smoothness of the overall uncertainty function allows
the ESO to converge at an exponential rate, which in turn boosts the performance of the overall control system.

We demonstrate the performance of the disturbance rejection control system \eqref{Eqn3.1} via some numerical simulations. In this simulation example presented below, we choose the Rayleigh numbers, $R_1=50, R_2=48,$ and $R_3=55$. The thermal and momentum coupling parameters are the same as before. The feedback gains are determined from the stability bounds stated in Theorem \ref{thm2.1}, which are $k_1=48$, $k_2=46$, and $k_3=50$ in this simulation. We simulate the disturbances via functions that depend on the local states of the subsystems, $i.e.$ $f_1=30 \sin(x_1z_1), f_2=x_2y_2 \cos(5t)$, and $f_3= 30 \sin(3t)$. The initial states of the ESO are set at $0$. Lastly, we choose $B= 60 Hz$ as the observer bandwidth through numerical experiments. It shows in Figure \ref{fig4.1}(a) that the system is unstable, even with large state feedback gains, when there are disturbances in the system with the disturbance rejection controller disabled. However, as soon as the ADRC controller is activated, the disturbances are eliminated effectively and the system is stabilized by the proportional controller with smaller gains that satisfy the stability bounds, see Figure \ref{fig4.1}(b). The performance of the ADRC design hinges upon the accuracy of the ESO when it is used to approximate the disturbance signals. It is shown in Figure \ref{fig6},  in this simulation, that the ESO is capable of quickly locking in with the unknown signal and tracking with the disturbance accordingly. Again, we list the steady state errors of the ADRC control system \eqref{Eqn3.1} in Table \ref{tab:steaderr2}. The results agree with the stability analysis given above. It is noted that the errors are not as small as those in Table \ref{tab:steaderr1}. This is because there exists numerical errors or discrepancies between the extended state from the ESO and the disturbance signal. Nevertheless, the numerical results provide strong evidence that the ADRC design in \eqref{Eqn3.1} is feasible to stabilize a high-dimensional multi-compartments dynamical system with disturbances.

\begin{figure}[H]
\centering
  \centering
\includegraphics[scale=0.48]{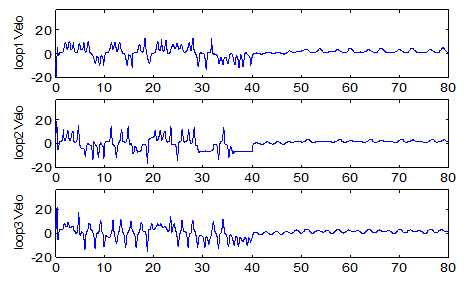}
  \centering
\includegraphics[scale=0.48]{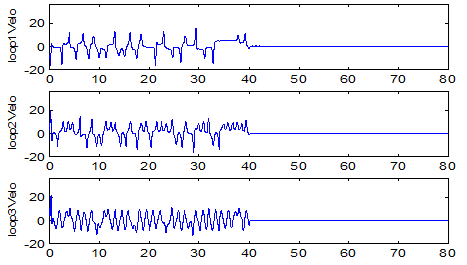}
\caption{Proportional controller activated at $t=40s:$ (\textbf{a}) destabilized fluid flow without ADRC due to existing disturbances; (\textbf{b}) stabilized flow with ADRC activated at $t=40s.$\label{fig4.1}}
\end{figure}

\begin{figure}[H]
\centering\includegraphics[width=8.0cm]{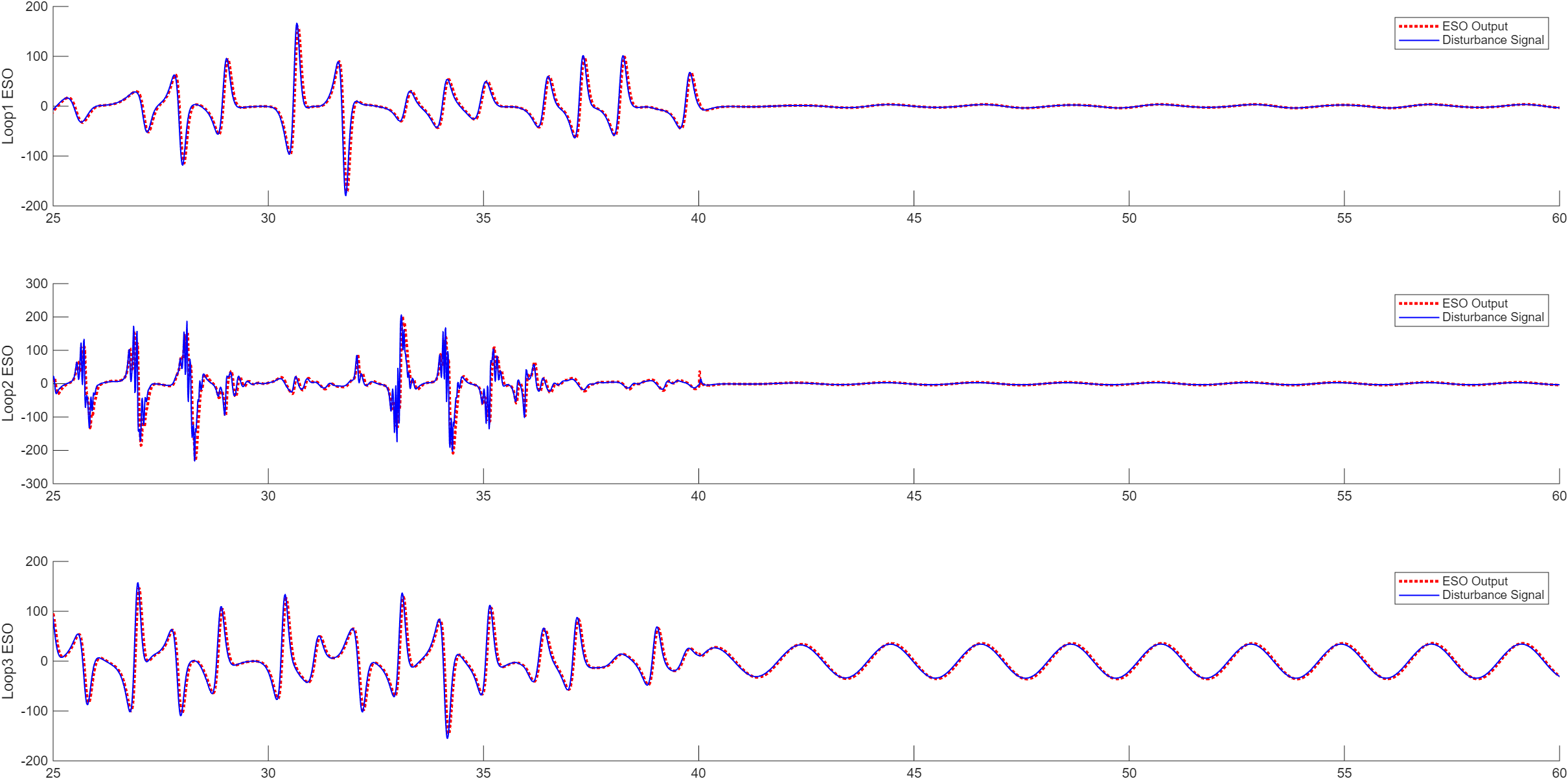}
\caption{ESO approximation of the disturbances in each loop\label{fig6}}
\end{figure}

\begin{table}[H]
\caption{Steady state error of the ADRC system}
\centering
\scriptsize                          
\setlength{\tabcolsep}{3pt}          
\resizebox{\linewidth}{!}{%
\begin{tabular}{|*{9}{c|}}
\hline
$x_1$ & $y_1$ & $z_1$ & $x_2$ & $y_2$ & $z_2$ & $x_3$ & $y_3$ & $z_3$ \\
\hline
   7.5766e-7 & 1.1124e-6 &  5.1001e-7 & 2.5049e-6 & 3.0616e-6 & 6.2432e-7 &   1.1747e-5 & 2.0806e-5 & 1.3621e-6 \\
   1.0290e-6 & 1.3422e-6 & 5.0601e-7 & 2.8179e-6 & 3.3278e-6 & 6.2032e-7  &   1.2142e-5 &  2.1140e-5 &  1.3710e-6\\
  1.2990e-6 & 1.5705e-6 & 5.0212e-7 & 3.1280e-6 &  3.5908e-6 &   6.1656e-7 &   1.2525e-5 & 2.1455e-5 & 1.3810e-6 \\
  1.5674e-6 &  1.7972e-6 & 4.9836e-7 & 3.4352e-6 & 3.8504e-6 & 6.1306e-7 &   1.2897e-5 & 2.1751e-5 & 1.3921e-6 \\
  1.8340e-6 & 2.0221e-6 & 4.9474e-7 & 3.7389e-6 & 4.1062e-6 & 6.0983e-7 &   1.3257e-5 & 2.2027e-5 & 1.4042e-6\\
  2.0986e-6 & 2.2449e-6 & 4.9126e-7 & 4.0391e-6 & 4.3581e-6 & 6.0688e-7 &   1.3605e-5 & 2.2282e-5 & 1.4172e-6\\
  2.3609e-6 & 2.4657e-6 & 4.8794e-7 & 4.3355e-6 & 4.6058e-6 & 6.0421e-7 &   1.3941e-5 & 2.2518e-5 & 1.4311e-6 \\
  2.6208e-6 & 2.6841e-6 &  4.8478e-7 & 4.6276e-6 &  4.8492e-6 &  6.0185e-7 &   1.4264e-5 & 2.2733e-5 & 1.4459e-6\\
  2.8780e-6 & 2.9000e-6 & 4.8180e-7 & 4.9154e-6 & 5.0880e-6 & 5.9980e-7 &   1.4573e-5 & 2.2927e-5 & 1.4615e-6\\
  3.1324e-6 & 3.1133e-6 & 4.7901e-7 & 5.1985e-6 &  5.3220e-6 &  5.9806e-7 &   1.4870e-5 & 2.3101e-5 & 1.4778e-6\\
\hline
\end{tabular}%
\label{tab:steaderr2}}
\end{table}

\section{Conclusions}\label{sec4}
In this paper, we applied two different decentralized controller designs to stabilize a
chaotic triangularly coupled triple loop thermosyphon system. In the case of modelbased control, we derived stability bounds on the feedback gains which depend on the Rayleigh numbers and the momentum coupling intensity parameter. When the system
parameters are unknown, we incorporated additional dynamic equations on the gains for
automatic gain searching. We proved the global asymptotic stability of the decentralized control system in each case. The second approach is based on the ADRC scheme when the system has disturbances. We implemented local extended state observers in the
subsystems to reconstruct the disturbances so that they can be eliminated during the ADRC process. We combined the model-based controller design and the decentralized ESO to stabilize the triple loop system with persistent disturbances. Numerical
simulation results further demonstrate the effectiveness of the proposed decentralized disturbance rejection control with a minimal number of proportional controllers. In this work, we only considered the existence of disturbances in the $y$-equations.  However, we can apply the proposed method in this paper to stabilize the system if the disturbances appear in the $x$-equations and $z$-equations. This is done by either implementing additional ADRC controllers in those equations or by way of ADRC-based backstepping control.   Due to
the simplicity of the proposed controller design, we expect that our results may have practical applications in some large-scale interconnected systems with uncertainties.


\end{document}